\documentclass[11pt]{amsart}

\usepackage{amsmath,amssymb,amsthm,graphics,amscd,mathrsfs}
\usepackage{amscd, amsfonts, amssymb, amsthm}
\usepackage{parskip, fullpage, verbatim}
\usepackage[colorlinks, breaklinks, linkcolor=blue]{hyperref} 
\usepackage{breakurl}
\usepackage{array}
\usepackage{latexsym}
\usepackage{enumerate}
\usepackage{longtable}
\usepackage{graphicx, subcaption}

\usepackage{etoolbox}
\makeatletter
\patchcmd\deferred@thm@head
  {\addvspace{-\parskip}}
  {}
  {}{\typeout{\string\deferred@thm@head patch failed!}}
\makeatletter

\newcommand\CA{{\mathscr A}} 
\newcommand\CB{{\mathcal B}}

\newcommand\CF{{\mathcal F}}

\newcommand\CL{{\mathcal L}}
\newcommand\CM{{\mathcal M}}

\newcommand\BBC{{\mathbb C}}
\newcommand\BBK{{\mathbb K}}
\newcommand\BBN{{\mathbb N}}
\newcommand\BBQ{{\mathbb Q}}

\newcommand\codim{\operatorname{codim}}

\newcommand\Der{{\operatorname{Der}}}

\newcommand\rk{\operatorname{rk}}

\newcommand\cl{\operatorname{cl}}

\numberwithin{equation}{section}

\theoremstyle{plain}

\newtheorem{theorem}[equation]{Theorem}

\newtheorem{corollary}[equation]{Corollary}
\newtheorem{proposition}[equation]{Proposition}
\theoremstyle{definition}
\newtheorem{definition}[equation]{Definition}
\newtheorem{remark}[equation]{Remark}

\newtheorem{example}[equation]{Example}

\begin{document}

\title[Combinatorially formal arrangements\\ are not determined by their points and lines]
{Combinatorially formal arrangements\\ are not determined by their points and lines}

\author[T. M\"oller]{Tilman M\"oller}
\address
{Fakult\"at f\"ur Mathematik,
Ruhr-Universit\"at Bochum,
D-44780 Bochum, Germany}
\email{tilman.moeller@rub.de}

\allowdisplaybreaks

\begin{abstract}
An arrangement of hyperplanes is called formal, if the relations between the hyperplanes are generated by relations in codimension 2. Formality is not a combinatorial property, raising the question for a characterization for combinatorial formality. A sufficient condition for this is if the underlying matroid has no proper lift with the same points and lines. We present an example of a matroid with such a lift but no non-formal realization, thus showing that above condition is not necessary for combinatorial formality.
\end{abstract}
\maketitle

\section{Introduction}

Let $\BBK$ be a field. An \emph{arrangement} $\CA$ is a finite collection of linear subspaces of $V=\BBK^\ell$ of codimension 1. 
Each hyperplane $H\in \CA$ is given as the kernel of a linear functional $\alpha_H\in V^*$ that is unique up to a scalar. 
Let $L(\CA)$ be the collection of all nonempty intersections of hyperplanes in $\CA$. 
We require $V\in L(\CA)$ as well. 
The set $L(\CA)$ is ordered by reverse inclusion and ranked by $r(X)=\codim X$ for $X\in L(\CA)$. 
In fact, $L(\CA)$ has the structure of a geometric lattice, called the \emph{lattice of flats}. 
It contains the combinatorial data of the arrangement $\CA$ and defines the underlying matroid $\CM(\CA)$. 
Two arrangements are called \emph{(combinatorially) isomorphic} if their underlying matroids are equal up to isomorphism. 
Any property that is invariant under such an isomorphism is called combinatorial.\\
Consider the linear map $\Phi:\BBK^{\CA}:=\bigoplus\limits_{H\in \CA}\BBK e_H \rightarrow V^*$ defined by $\Phi(e_H)=\alpha_H$. 
If $\ker \Phi$ is generated by its elements of weight at most three, 
i.e.~vectors with 3 or fewer nonzero entries, 
$\CA$ is called  \emph{formal}, see \cite{falkrand:homotopy1}. In \cite{yuzv:formality}, Yuzvinsky showed that formality is not combinatorial, so it is natural to ask whether matroids that admit only formal arrangements can be characterized intrinsically. A matroid is called \emph{taut} if it is not a proper quotient of a matroid with the same points and lines, see Definition \ref{def:taut}. An arrangement with an underlying taut matroid is necessarily formal. For a survey on this topic, see \cite[Ch.~3]{falk:formality}. In \emph{loc.~cit.}, Falk asked whether there is a non-taut matroid that only admits formal arrangements as realizations. In this paper we give such an example, thus showing the following.
\begin{theorem}
\label{thm:main-theorem}
There is a realizable matroid $M$ that is not taut such that every realization of $M$ is formal.
\end{theorem}
\textbf{Acknowledgements:} The author thanks Michael Falk for helpful discussions regarding the content of this paper.
\pagebreak
\section{Recollections and Preliminaries}
Let $E$ be a finite set. A \emph{matroid} $M$ on the \emph{ground set} $E$ is a collection $\CB$ of subsets of $E$ subject to
\begin{enumerate}
\item[(i)] $\CB\neq \emptyset$ and
\item[(ii)] for all $B,B'\in \CB$ and every $f\in B'\setminus B$ there is an $e\in B\setminus B'$ such that $\left(B'\setminus\{f\}\right)\cup\{e\}\in \CB$.
\end{enumerate}
An element $B\in \CB$ is called a \emph{basis} or a \emph{base} of $M$. Note that all bases have the same cardinality. Any subset of a base is an \emph{independent set} of $M$. Subsets of $E$ that are not independent are \emph{dependent}, the minimal dependent sets are called \emph{circuits}.

The \emph{rank} $\rk(X)$ of a subset $X\subset E$ is the size of a maximal independet subset of $X$, and the rank of $M$ is defined by $\rk(M)=\rk(E)$. There is a notion of closure on $M$ sending subsets to their maximal supersets of the same rank, i.e.~
\[
\cl(X):=\overline X:=\left\{e\in E\mid \rk(X)=\rk(X\cup\{e\})\right\}.
\]
A set $X\subset E$ is called \emph{closed} or a \emph{flat of $M$} if $X=\overline X$. The set $\CL=\CL(M)$ of all flats is partially ordered by inclusion. It has the structure of a geometric lattice and is called the \emph{lattice of flats}. Flats of rank one (resp.~two) are called \emph{points} (resp.~\emph{lines}) of $M$. An element $e\in E$ that is dependent on its own is called a \emph{loop}, two dependent elements $\{i,j\}$ are called \emph{parallel}. A matroid is called \emph{simple} if it has no loops or parallel elements.
A matroid is completely determined by its bases, circuits, rank function, closure or the lattice of flats.

For ease of notation we write $\CL_k$ for the elements of $\CL$ of rank $k$, and $\CL_k^{>s}$ for flats of rank $k$ and cardinality greater than $s$. We call $\CL_{\rk(E)-1}$ the set of \emph{copoints of $M$}. In fact, the collection of copoints contains enough information to uniquely define the matroid.
\begin{definition}
\label{def:quotient}
Let $M,N$ be two matroids on the same ground set $E$. If any independent set of $M$ is independent in $N$, we call $M$ a \emph{weak map image} of $N$ and write $M\prec N$. If $M$ is a weak map image of $N$ and further $\CL(M)\subset \CL(N)$, we call $M$ a \emph{quotient} of $N$. Note that $\prec$ defines a partial order on the class of all matroids.
\end{definition}
Let $X\subset E$. The \emph{deletion of $X$ from $M$} is the matroid $M-X$ on the ground set $E \setminus X$. Its independent sets are the independent sets of $M$ disjoint from $X$. The \emph{contraction of $X$ from $M$} is the matroid $M/X$ on $E\setminus X$. Its circuits are the minimal non-empty sets in $\{C\setminus X\mid C\in \mathcal{C}(M)\}$. A \emph{minor of M} is a matroid that arises as a sequence of deletions and contractions of $M$.

Sometimes the dependencies in $M$ can be realized as the linear dependencies of a set of vectors. Let $\rk(M)=\ell$. If there is a set $A=\{v_1,\dots,v_n\}$ of vectors of $\BBK^\ell$ such that $B\in \CB$ if and only if $\{v_i\mid i\in B\}$ is a basis of $\BBK^\ell$, then $M$ is called \emph{$\BBK$-linear} and $A$ is called a \emph{realization} of $M$. Due to the next proposition, to show that a matroid $M$ is not realizable over a certain field $\BBK$, it suffices to find a minor of $M$ that is not realizable over $\BBK$.
\begin{proposition}[{\cite[Prop.~3.2.4]{oxley:matroids}}]
\label{prop:minorrep}
If a matroid is realizable over a field $\BBK$, then all its minors are as well.
\end{proposition}

\begin{example}
\label{ex:FanoMatroid}
Define $F_7$ as the matroid of rank $3$ on $E=\{0,\dots,6\}$ with non-trivial lines
\[
\CL_2^{>2}(F_7) = \{015,024,036,123,146,256,345\}
\]
and define $F_7^-$ as the matroid of rank $3$ on the same ground set $E$ with non-trivial lines
\[
\CL_2^{>2}(F_7^-) = \CL_2^{>2}(F_7) \setminus \{345\}.
\]
$F_7$ is called the \emph{Fano matroid} and $F_7^-$ is called the \emph{non-Fano matroid}. Pictures of the two matroids are given in Figure \ref{fig:FanoEx}. 
In the pictures, every point is a point of the matroid, 
and three points are connected by a line segment if the three points are contained in a flat of rank two.
Let $M\in \{F_7, F_7^-\}$ and let $A = (I_3 \mid X)$ be a representation of $M$ over a field $\BBK$, where $I_3$ is the $3\times 3$ identity matrix and $X$ is a $3\times 4$-matrix, such that the $i$-th column represents the element $i\in \{0,\dots,6\}$. Then
\[
X=
\left(
\begin{array}{llll}
0 & 1 & 1 & 1 \\
1 & 0 & 1 & 1 \\
1 & 1 & 0 & 1 \\
\end{array}
\right)
\]
and $M=F_7$ if and only if $\operatorname{char}(\BBK) = 2$ (cf.~\cite[Prop.~6.4.8]{oxley:matroids}). Thus, the Fano matroid is only realizable over fields of characteristic two and the non-Fano matroid is only realizable over fields of characteristic different from two.
\end{example}

\begin{figure}[]
\centering
\begin{subfigure}[b]{0.45\linewidth}
\includegraphics[width=\linewidth]{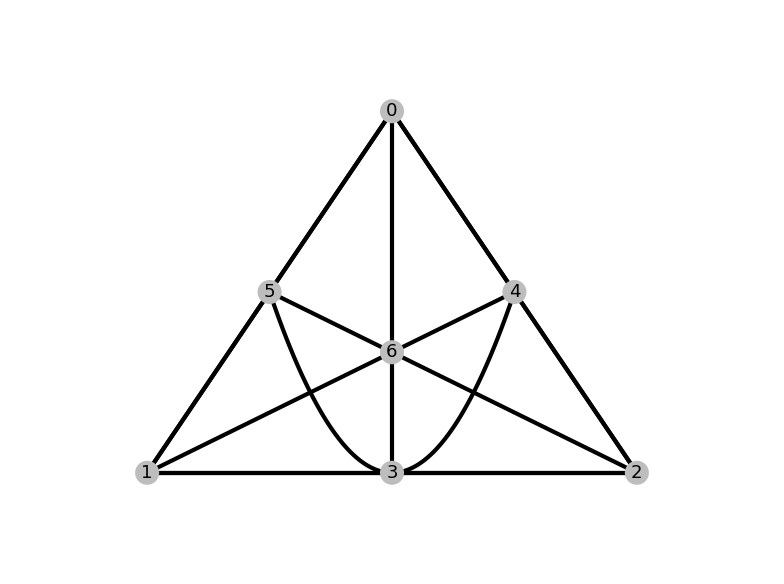}
\caption{Fano matroid $F_7$}
\end{subfigure}
\begin{subfigure}[b]{0.45\linewidth}
\includegraphics[width=\linewidth]{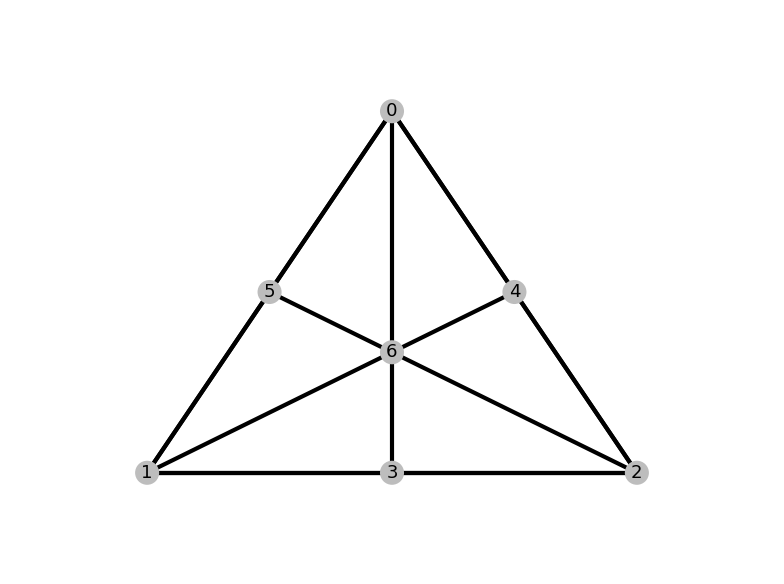}
\caption{non-Fano matroid $F_7^-$}
\end{subfigure}
\caption{Matroids from Example \ref{ex:FanoMatroid}}
\label{fig:FanoEx}
\end{figure}

Let $V=\BBK^\ell$. A finite set $\CA=\{H_1,\dots,H_n\}$ with $H_1,\dots,H_n$ (linear) hyperplanes in $V$ is called a \emph{(central) arrangement}. 
Choose linear forms $\alpha_i\in V^*$ such that $\ker \alpha_i = H_i$. 
Let $M=\CM(\CA)$ be the $\BBK$-linear matroid realized by $(\alpha_1,\dots,\alpha_n)$. It contains the combinatorial data of the arrangement $\CA$. The lattice $\CL$ of $\CM(\CA)$ is canonically isomorphic to the collection of all nonempty intersections of hyperplanes of $\CA$. The \emph{rank} of $\CA$ is the codimension of the intersection of all its hyperplanes, i.e.~$r(\CA)=\codim(\bigcap H)$. It coincides with the rank of the underlying matroid.

Probably the most studied properties of arrangements are freeness and asphericity. 
Let $S=S(V^*)$ be the symmetric algebra of $V^*$. 
The product $Q(\CA) = \prod\limits_{i=1}^n \alpha_i\in S$ is called the \emph{defining polynomial} of $\CA$. 
Note that after chosing a basis $(e_1,\dots,e_\ell)$ of $V$ 
and a dual basis $(x_1,\dots,x_\ell)$ of $V^*$, we have $S \cong \BBK[x_1,\dots,x_\ell]$.
Let 
$$\Der(S)=\{\theta:S\rightarrow S\mid \theta(fg)=f\theta(g)+g\theta(f) \text{ for all } f,g\in S\}$$
be the $S$-module of formal derivations of $S$.
An arrangement is called \emph{free} if the $S$-module 
\[
D(\CA)=\{\theta \in \Der(S) \mid \theta(Q(\CA))\in Q(\CA)S\}
\]
is free.
A complex arrangement is called \emph{aspherical} if the complement $\BBC^\ell \setminus \bigcup H$ is a $K(\pi,1)$-space. Whether freeness and asphericity are combinatorial properties are important open problems in arrangement theory.
A comprehensive summary about arrangement theory can be found in \cite{orlikterao:arrangements}.
\begin{definition}
Let $\BBK^\CA=\bigoplus\limits_{H\in \CA} \BBK e_H$ be the vector space with basis indexed by the hyperplanes in $\CA$ and define $\Phi:\BBK^\CA\rightarrow V^*$ by $\Phi(e_H)=\alpha_H$ and linear extension. Let $F\subset \ker\Phi$ be the subspace generated by all elements of $\ker\Phi$ with at most three nonzero entries. Then $\CA$ is called \emph{formal} if $F=\ker\Phi$.
\end{definition}
The notion of formality first appeared in \cite{falkrand:homotopy1}, where it was introduced as a necessary condition for asphericity. Later, Yuzvinsky showed that it is also necessary for free arrangements.
\begin{theorem}
Let $\CA$ be an arrangement.
\begin{enumerate}
\item[(i)]\cite{falkrand:homotopy1} If $\CA$ is aspherical, then it is formal.
\item[(ii)]\cite{yuzv:formality} If $\CA$ is free, then it is formal.
\end{enumerate}
\end{theorem}
Formality is not a combinatorial property. The first example in the literature is due to Yuzvinsky.
\begin{example}{\cite[Ex.~2.2]{yuzv:formality}}
\label{ex:yuzv}
Define $Q_0=xyz(x+y+z)(2x+y+z)(2x+3y+z)(2x+3y+4z)$ and define arrangements $\CA_1$ and $\CA_2$ by $Q(\CA_1)=Q_0\cdot (3x+5z)(3x+4y+5z)$ and $Q(\CA_2)=Q_0\cdot (x+3z)(x+2y+3z)$. Then the underlying matroids of $\CA_1$ and $\CA_2$ are the same, but $\CA_1$ is formal while $\CA_2$ is not.
\end{example}
Since formality is not combinatorial, it makes sense to ask for a property of the matroid such that each $\BBK$-representation of it is formal.
\begin{remark}
\label{rem:formality}
Consider the map $\pi:V\rightarrow \BBK^\CA$ defined by $\pi(x) = (\alpha_1(x),\dots,\alpha_n(x))^T$. If $y=(y_1,\dots,y_n)\in \ker\Phi$, consider the scalar product $\pi(x) y=\sum y_i \alpha_i(x) = \Phi(y)(x)=0$, so $\operatorname{Im} \pi = \ker\Phi^\bot$. Thus $\ker \Phi$ contains all the information of $\CA$ and $\CA$ can be reconstructed via 
\[
\CA \cong \{\ker\Phi^\bot\cap\{x_i=0\}\mid i=1,\dots,n\}.
\]
The same construction for $F$ yields
\(
\CA_F:=\{F^\bot\cap\{x_i=0\}\mid i=1,\dots,n\},
\)
the \emph{formalization} of $\CA$. Clearly, $r(\CA)\leq r(\CA_F)$ and $r(\CA)=r(\CA_F)$ if and only if $\CA$ is formal. Furthermore, it is easy to see that $\CM(\CA)$ is a quotient of $\CM(\CA_F)$ with the same points and lines.
\end{remark}
\begin{definition}{\cite[Def.~3.5]{falk:formality}}
\label{def:taut}
A matroid $M$ is called \emph{taut} if it is not a quotient of any matroid of higher rank with the same points and lines.
\end{definition}
Because of Remark \ref{rem:formality}, a $\BBK$-representation of a taut matroid is always formal, since its formalization cannot admit it as a proper quotient. This paper is dedicated to showing that the reverse implication is false, which answers a question raised by Falk in \cite{falk:formality}. To validate our claim, we use the theory established  in \cite{crapo:erectinggeometries} about erections of matroids.
\begin{definition}
Let $M$ be a matroid on $E$ of rank $r>1$. The \emph{truncation of $M$} is the (unique) matroid $T$ of rank $r-1$ with $\CL(T)=\CL_{<r}(M)\cup E$. Thus, $T \prec M$. A matroid $N$ is an \emph{erection of $M$} if the truncation of $N$ is isomorphic to $M$. We further say $M$ is the \emph{trivial erection} of itself.
\end{definition}
Note that while the truncation is uniquely defined, there can be many erections of a matroid. Let $\mathcal{E}(M)$ be the collection of erections of $M$.
\begin{theorem}{\cite[Thm.~9]{crapo:erectinggeometries}}
Let $M$ be a matroid, then the set $\mathcal{E}(M)$ together with the relation $\prec$ from Definition \ref{def:quotient} has the structure of a geometric lattice. Its minimal element is the trivial erection $M$. Define the \emph{free erection of $M$} as the maximal element of $\mathcal{E}(M)$.
\end{theorem}
Let $M$ be a matroid on $E$ and let $k\in \BBN$. A subset $X\subset E$ is \emph{$k$-closed} if it contains the closures of all its $k$-element subsets. We say \emph{$X$ spans $M$} if $\overline{X} = M$.
The following theorem characterizes erections of $M$ by their copoints.
\begin{theorem}{\cite[Thm.~2]{crapo:erectinggeometries}}
\label{thm:crapo-charac}
Let $M$ be a matroid of rank $r$ on $E$. A set $\CF$ of subsets (called blocks) of $E$ is the set of copoints of an erection of $M$ if and only if
\begin{enumerate}[(i)]
\item each block spans $M$;
\item each block is $(r-1)$-closed;
\item each basis of $M$ is contained in a unique block.
\end{enumerate}
\end{theorem}
\section{Proof of Theorem \ref{thm:main-theorem}}
\begin{figure}[b]
\includegraphics[width=0.65\linewidth]{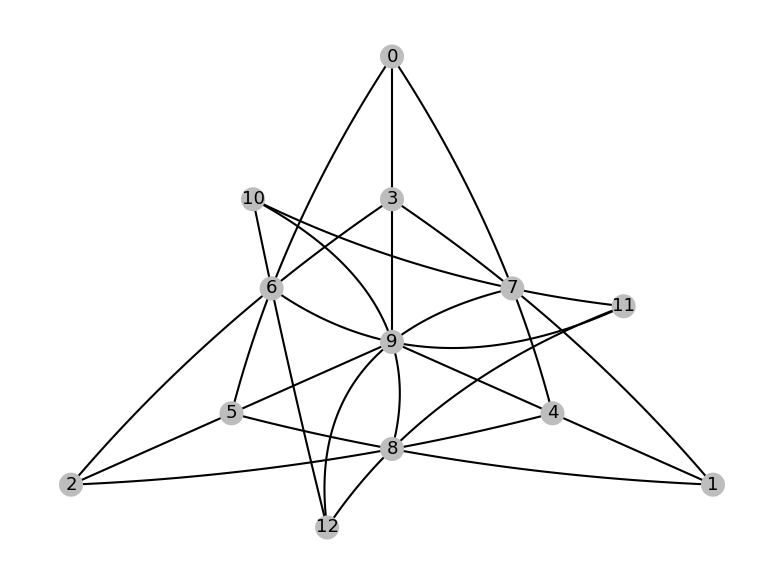}
\caption{The matroid $M$.}
\label{fig:matroid}
\end{figure}
Let $M$ be the simple matroid on $E=\{0,\dots,12\}$ of rank 3 with the following nontrivial flats in rank 2:
\[
\CL_2^{>2}(M)=\left\{
\begin{array}{ccccc}
\{0,3,9\}, & \{0,4,7\}, & \{0,5,6\}, & \{8,9,10\}, & \{7,10,11\}, \\
\{1,4,9\}, & \{1,3,7\}, & \{1,5,8\}, & \{6,9,11\}, & \{6,10,12\}, \\
\{2,5,9\}, & \{2,3,6\}, & \{2,4,8\}, & \{7,9,12\}, & \{8,11,12\} \\
\end{array}
\right\}.
\]
For a picture of $M$ see Figure \ref{fig:matroid}.
Note that for $X=\{0,\dots,8\}\subset E$, $M$ contains the underlying matroid of Example \ref{ex:yuzv} as a minor. For the subset $Y=\{6,\dots,12\}$ the non-Fano matroid $F_7^-$ is also a minor of $M$, see Figure \ref{fig:minors}.

\begin{figure}[]
\centering
\begin{subfigure}[b]{0.45\linewidth}
\includegraphics[width=\linewidth]{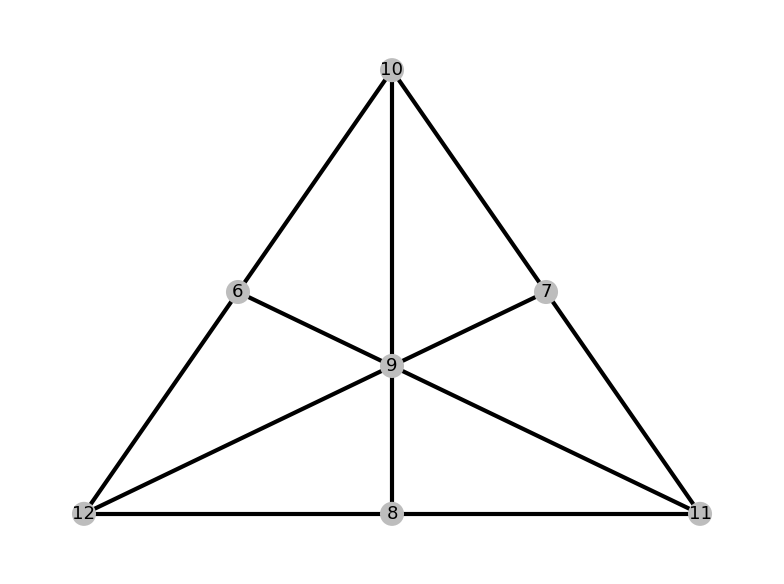}
\caption{non-Fano matroid $F_7^-$}
\end{subfigure}
\begin{subfigure}[b]{0.45\linewidth}
\includegraphics[width=\linewidth]{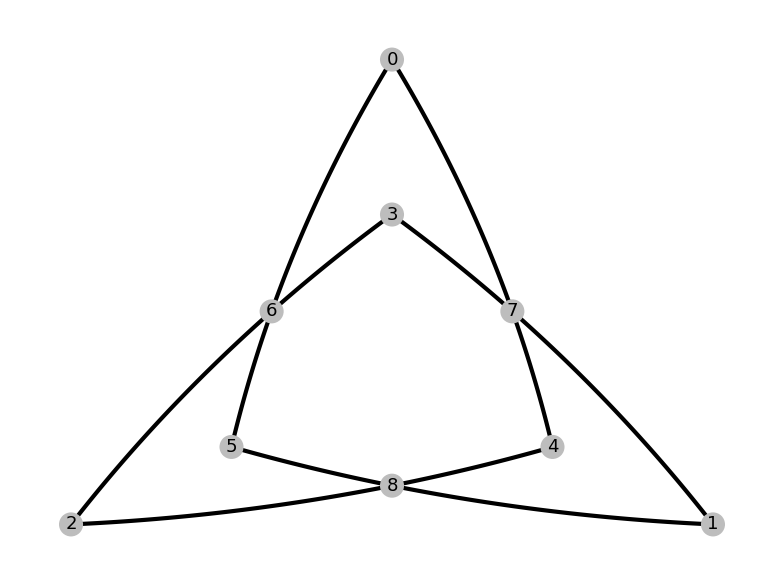}
\caption{matroid from Example \ref{ex:yuzv}}
\end{subfigure}
\caption{Two minors of $M$.}
\label{fig:minors}
\end{figure}

A realization of $M$ over $\BBQ$ is given by
\[
A=\left(
\begin{array}{rrrrrrrrrrrrr}
1 & 4 & 4 & 8 & 4 & 2 & 1 & 0 & 0 & 4 & 4 & 4 & 4 \\
1 & -2 & 1 & -1 & 1 & -1 & 0 & 1 & 0 & -5 & -5 & 5 & 5 \\
1 & 5 & -10 & 10 & 4 & -1 & 0 & 0 & 1 & 6 & -6 & -6 & 6 \\
\end{array}
\right),
\]
where the $i$-th column of $A$ belongs to the element $i\in E$.
We mention that no realization of $M$ can be free (since its characteristic polynomial does not factor). Furthermore, as a complex arrangement $A$ is not aspherical, since it has a \emph{simple triangle} (cf.~\cite[Cor.~3.3]{falkrand:homotopy1}). We have not verified whether other realizations of $M$ are not aspherical, yet we mention that there are realizations of $M$ that do not admit a simple triangle.

Next we define the matroid $N$ of rank 4 with the same points and lines as $M$. The non-trivial flats of rank 3 are given by 
\[
\CL_3^{>3}(N) = \left\{
\begin{array}{llll}
\{0, 1, 3, 4, 7, 9, 12\},	& \{0, 4, 5, 6, 7\},	& \{0, 4, 7, 10, 11\},	& \{0, 8, 11, 12\},	\\
\{0, 2, 3, 5, 6, 9, 11\},	& \{1, 2, 3, 6, 7\},	& \{1, 3, 7, 10, 11\},	& \{1, 6, 10, 12\},	\\
\{1, 2, 4, 5, 8, 9, 10\},	& \{0, 1, 5, 6, 8\},	& \{0, 5, 6, 10, 12\},	& \{2, 7, 10, 11\},	\\
\{6, 7, 8, 9, 10, 11, 12\},	& \{2, 3, 4, 6, 8\},	& \{2, 3, 6, 10, 12\},	& \{3, 8, 11, 12\},	\\
\{0, 3, 8, 9, 10\},			& \{0, 2, 4, 7, 8\},	& \{1, 5, 8, 11, 12\},	& \{4, 6, 10, 12\},	\\
\{1, 4, 6, 9, 11\},			& \{1, 3, 5, 7, 8\},	& \{2, 4, 8, 11, 12\},	& \{5, 7, 10, 11\},	\\
\{2, 5, 7, 9, 12\}
\end{array}
\right\}.
\]
Furthermore, $\CL_3(N)$ also contains every three-element subset of $E$ that is not in $\CL_2(N)=\CL_2(M)$ or a subset of a flat in $\CL_3^{>3}(N)$, i.e.
\[
\CL^{=3}_3(N)= \left\{
\begin{array}{llll}
\{0,1,10\}, & \{0,2,12\}, & \{3,4,10\}, & \{3,5,12\},\\
\{0,2,10\}, & \{1,2,12\}, & \{3,5,10\}, & \{4,5,12\},\\
\{0,1,11\}, & \{0,1,2\}, & \{3,4,11\}, & \{3,4,5\},\\
\{1,2,11\}, & & 			\{4,5,11\}&
\end{array}
\right\}
\]
Note that $\CL_3(N)$ satisfies the conditions from Theorem \ref{thm:crapo-charac}, so $N$ is an erection of $M$. This implies that $M$ is not taut. Next we show that $N$ is the only non-trivial erection of $M$. 
\begin{proposition}
\label{prop:onlyN}
We have $\mathcal{E}(M) = \{M,N\}$.
\end{proposition}
\begin{proof}
Suppose $N'\neq M$ is an erection of $M$. Then, the copoints of $N'$ have to fulfil the conditions (i)--(iii) from Theorem \ref{thm:crapo-charac}. The $2$-closed sets with respect to $M$ that span $M$ are precisely $\CL_3(N)\cup S$, where
\[
S=\left\{
\begin{array}{llll}
\{0,1,12\}, & \{3,4,12\}, & \{0,1,2,10\}, & \{3,4,5,10\},\\
\{0,2,11\}, & \{3,5,11\}, & \{0,1,2,11\}, & \{3,4,5,11\},\\
\{1,2,10\}, & \{4,5,10\}, & \{0,1,2,12\}, & \{3,4,5,12\},\\
\{6,7,8\} &&&
\end{array}
\right\}.
\]
We argue that no element of $S$ can be a copoint of $N'$, thus implying our statement. First assume that $X\in S$ is of cardinality $3$. Then $X$ is a basis of $M$, thus by Theorem \ref{thm:crapo-charac}(iii) there is a unique block $Z\in \CL_3(N)$ with $X\subset Z$. So if $X$ is a copoint of $N'$, then $Z$ is not. Now observe that for every choice of $X$, there are bases $B$ of $M$ with $B\subset Z$ that are not a subset of any other possible block in $\CL_3(N)\cup S$. For completeness, we specify a base for each of the seven choices for $X$:
\begin{itemize}
\item if $X=\{0,1,12\}$ or $X=\{3,4,12\}$, then $Z=\{0,1,3,4,7,9,12\}$ and $B=\{0,1,3\}$.
\item if $X=\{0,2,11\}$ or $X=\{3,5,11\}$, then $Z=\{0,2,3,5,6,9,11\}$ and $B=\{0,2,3\}$.
\item if $X=\{1,2,10\}$ or $X=\{4,5,10\}$, then $Z=\{1,2,4,5,8,9,10\}$ and $B=\{1,2,4\}$.
\item if $X=\{6,7,8\}$, then $Z=\{6,7,8,9,10,11,12\}$ and $B=\{7,8,9\}$.
\end{itemize}
Finally assume that $Y\in S$ is of cardinality $4$. This case reduces to the first one since there always is a $X\in S$ with $X\subsetneq Y$, so with the same reasoning as before, $Y$ is not a copoint of $N'$. Thus $N'=N$.
\end{proof}

Since $N$ is the only non-trivial erection of $M$, if $\CA$ is a non-formal arrangement with $M=\CM(\CA)$, then $N$ must be a (potentially trivial) quotient of $\CM(\CA_F)$. Hence, if $N$ is not realizable, any arrangement realizing $M$ must be formal.
\begin{proposition}
The matroid $N$ is not realizable over any field $\BBK$.
\begin{proof}
First, observe that the deletion $N-\{0,\dots,5\}$ is the non-Fano matroid $F_7^-$, so by Proposition \ref{prop:minorrep} and Example \ref{ex:FanoMatroid}, $N$ is realizable only over fields of characteristic different from $2$. Furthermore, it turns out that $F_7$ is a minor of $N$ as well. To see this, consider the contraction $P=N/\{6\}$ and consider parallel elements as a single point. The points of $P$ then are
\[
\CL_1(P) = \left\{
[0,5],\,[1],\,[2,3],\,[4],\,[7],\,[8],\,[9,11],\,[10,12]
\right\}
\]
and the non-trivial lines of $P$ are
\[
\CL_2^{>2}(P)=
\left\{
\begin{array}{ll}
\{[0,5],[1],[8]\}, & \{[1],[2,3],[7]\}, \\
\{[0,5],[2,3],[9,11]\}, & \{[1],[4],[9,11]\}, \\
\{[0,5],[4],[7]\}, & \{[2,3],[4],[8]\}, \\
\{[7],[8],[9,11],[10,12]\} & 
\end{array}
\right\}.
\]
Thus, $F_7=P-\{[10,12]\}$ is a minor of $N$, so again by Proposition \ref{prop:minorrep} and Example \ref{ex:FanoMatroid}, $N$ is not realizable over any characteristic.
\end{proof}
\end{proposition}
\begin{corollary}
Every realization of the matroid $M$ is formal.
\end{corollary}

This completes the proof of Theorem \ref{thm:main-theorem}.

\bibliographystyle{alpha}
\bibliography{literature.bib}

\begin{thebibliography}{Yuz93}

\bibitem[Cra70]{crapo:erectinggeometries}
H.~Crapo.
\newblock Erecting geometries.
\newblock {\em Ann.~New York Acad.~Sci.}, 175:89--92, 1970.

\bibitem[Fal02]{falk:formality}
M.~Falk.
\newblock Line-closed matroids, quadratic algebras, and formal arrangements.
\newblock {\em Advances in Applied Mathematics}, (28):250–271, 2002.

\bibitem[FR86]{falkrand:homotopy1}
M.~Falk and R.~Randell.
\newblock On the homotopy theory of arrangements.
\newblock {\em Advanced Studies in Pure Mathematics}, 8:101--124, 1986.

\bibitem[OT92]{orlikterao:arrangements}
P.~Orlik and H.~Terao.
\newblock {\em Arrangements of Hyperplanes}.
\newblock Springer-Verlag, 1992.

\bibitem[Oxl92]{oxley:matroids}
J.~G. Oxley.
\newblock {\em Matroid theory}.
\newblock Oxford University Press, 1992.

\bibitem[Yuz93]{yuzv:formality}
S.~Yuzvinsky.
\newblock The first two obstructions to the freeness of arrangements.
\newblock {\em Transactions of the American Mathematical Society}, 335(1),
  1993.

\end{thebibliography}

\end{document}